\documentclass[12pt]{amsart}
\usepackage{amsthm}
\usepackage{amssymb}
\usepackage{amsmath}
\usepackage{amsfonts}
\usepackage{graphicx} 
\usepackage{mathtools}
\usepackage{soul}
\usepackage{tikz}
\usepackage{ytableau,tikz,varwidth}
\usepackage{latexsym,amssymb,amsmath,hyperref,amsthm,amsfonts,
caption,tikz,comment}
\usetikzlibrary{calc}
\usepackage{tikz-cd}
\usepackage{enumitem}
\usetikzlibrary{graphs}
\usetikzlibrary{graphs.standard}
\usetikzlibrary{shapes.misc}
\usetikzlibrary {shapes.geometric} 
\usepackage{amsrefs}
\usepackage{latexsym,hyperref,caption,subcaption,comment}

\numberwithin{equation}{section}
\tikzset{
  curarrow/.style={
  rounded corners=8pt,
  execute at begin to={every node/.style={fill=red}},
    to path={-- ([xshift=-50pt]\tikztostart.center)
    |- (#1) node[fill=white] {$\scriptstyle \delta$}
    -| ([xshift=50pt]\tikztotarget.center)
    -- (\tikztotarget)}\
    }
}

\tikzcdset{diagrams={arrows={shorten >=-0.5ex,shorten <=-0.5ex}}}

\newtheorem{theorem}{Theorem}[section]
\newtheorem{corollary}[theorem]{Corollary}
\newtheorem{lemma}[theorem]{Lemma}
\newtheorem{proposition}[theorem]{Proposition}

\theoremstyle{definition}

\newtheorem{example}[theorem]{Example}
\newtheorem{remark}[theorem]{Remark}

\def\opn#1#2{\def#1{\operatorname{#2}}} 
\opn\Cl{Cl} \opn\pdim{pdim} \opn\Im{Im} \opn\Ker{Ker} \opn\ini{in} \opn\typ{type}

\begin{document}

\title{$h$-vectors of edge rings of odd-cycle compositions}
\author{Kieran Bhaskara}
\address{Department of Mathematics and Statistics, McMaster University, Hamilton, Ontario L8S 4L8, Canada}
\email{kieran.bhaskara@mcmaster.ca}
\author{Akihiro Higashitani}
\address{Department of Pure and Applied Mathematics, Graduate School of Information Science and Technology, Osaka University, Suita, Osaka 565-0871, Japan}
\email{higashitani@ist.osaka-u.ac.jp}
\author{Nayana Shibu Deepthi}
\address{Department of Pure and Applied Mathematics, Graduate School of Information Science and Technology, Osaka University, Suita, Osaka 565-0871, Japan}
\email{nayanasd@ist.osaka-u.ac.jp}


\begin{abstract}
Let $\mathbb{K}[G]$ be the edge ring of a finite simple graph $G$. Investigating properties of the $h$-vector of $\mathbb{K}[G]$ is of great interest in combinatorial commutative algebra. However, there are few families of graphs for which the $h$-vector has been explicitly determined. In this paper, we compute the $h$-vectors of a certain family of graphs that satisfy the odd-cycle condition, generalizing a result of the second and third named authors. As a corollary, we obtain a characterization of  the graphs in this family whose edge rings are almost Gorenstein.
\end{abstract}

\maketitle

\section{Introduction}\label{sec:intro}
A major theme in commutative algebra is the study of invariants of ideals belonging to specific families, such as those arising from combinatorics. Indeed, much work has been performed towards understanding invariants of graph-theoretic ideals such as edge ideals, toric ideals, and binomial edge ideals by exploiting the combinatorics of the underlying graph. (For examples, see \cites{LZ,BOVT, FS, HaH, EH, HH, HKKMV, HHKO, KK, KKS, Kumar, VV, MR3959482, HKO, Kanno, BVT}.) The $h$-polynomial, in particular, has been the subject of much attention in recent studies; we  briefly recall its definition. Given a polynomial ring $S=\mathbb{K} [x_1,\dots,x_m]$ with the standard grading over a field $\mathbb{K}$, and a homogeneous ideal $I$ of $S$, the \textit{Hilbert series} of $S/I$ is the formal power series $\textup{HS}(S/I;t) = \sum_{i\geq0} [\dim_\mathbb{K}(S/I)_i]t^i$
where $\dim_\mathbb{K}(S/I)_i$ is the dimension of the $i$-th graded piece of $S/I$. By the Hilbert--Serre Theorem~\cite{Villarreal-book}*{Theorem 5.1.4}, there exists a unique polynomial $h(S/I;t) \in \mathbb{Z}[t]$, called the {\it $h$-polynomial} of $S/I$, such that $\textup{HS}(S/I;t)$ can be written as \[
    \textup{HS}(S/I;t)=\frac{h(S/I;t)}{(1-t)^{\dim(S/I)}}
    \] with $h_{S/I}(1) \neq 0$. If we write $h(S/I;t)=h_0+h_1t+\cdots +h_st^s$ with $h_s \neq 0$, we call the sequence $(h_0,h_1,\dots,h_s)$ of coefficients of $h(S/I;t)$ the \textit{$h$-vector} of $S/I$.
    
Many authors have studied properties of  $h$-polynomials of ideals defined from graphs \cites{FKVT, HKMT, HKMVT, HM3, HMVT}. However, for toric ideals $I_G$, there are few families of graphs for which the $h$-polynomial of $S/I_G$ has been explicitly described. As far as we know, the $h$-polynomials (or their counterparts) of the edge rings of the following families of graphs have been computed: complete bipartite graphs, \cite{Villarreal-book}*{Proposition 10.6.3}, complete graphs \cite{Villarreal-book}*{Proposition 10.6.10}, complete multipartite graphs \cite{OH00}*{Theorem 2.6}, 
 bipartite graphs (via interior polynomials) \cite{KP}, and a family of graphs composed of a complete bipartite graph and a cycle graph \cite{Galetto}.
Note that if the edge ring of a graph is  normal, then the Hilbert function (resp. $h$-vector) of the edge ring
agrees with the Ehrhart polynomial (resp. $h^*$-vector) of the edge polytope arising from the graph; this fact is used to obtain several of the above results.

In the previous paper \cite{HN}, the second and third named authors determined the $h$-polynomials of a certain family of non-bipartite graphs $\mathcal{G}_n$, consisting of triangles that share a single common vertex (see Figure~\ref{fig1}). For this family, the authors were then able to show that the edge rings $\mathbb{K}[\mathcal{G}_n]$ are almost Gorenstein. As a step towards generalizing this result, the first goal of this paper is to compute the $h$-polynomials of a related family of graphs, containing those studied in \cite{HN}.

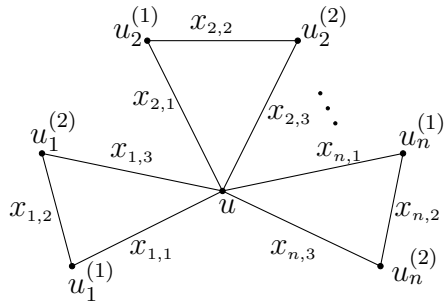
\begin{figure}[h]
\centering
\begin{tikzpicture}
\draw[black, thin] (3,1) -- (5,2) -- (2.6,2.5)-- cycle;
\draw[black, thin] (5,2) -- (4,4) -- (6,4)-- cycle;
\draw[black, thin] (5,2) -- (7.4,2.5) -- (7.1,1)-- cycle;
\filldraw [black] (3,1) circle (1pt);
\filldraw [black] (5,2) circle (1pt);
\filldraw [black] (2.6,2.5) circle (1pt);
\filldraw [black] (4,4) circle (1pt);
\filldraw [black] (6,4) circle (1pt);
\filldraw [black] (7.4,2.5) circle (1pt);
\filldraw [black] (7.1,1) circle (1pt);
\filldraw [black] (6.5,2.9) circle (0.6pt);
\filldraw [black] (6.3,3.3) circle (0.6pt);
\filldraw [black] (6.4,3.1) circle (0.6pt);
\filldraw [black] (4.8,1.8) node[anchor=west] {$u$};
\filldraw [black] (2.8,0.8) node[anchor=west] {$u^{(1)}_{1}$};
\filldraw [black] (3.4,4.2) node[anchor=west] {$u^{(1)}_{2}$};
\filldraw [black] (3.8,2.2) node[anchor=south] {$x_{{\scriptscriptstyle 1,3}}$};
\filldraw [black] (2.9,1.7) node[anchor=east] {$x_{{\scriptscriptstyle 1,2}}$};
\filldraw [black] (4.5,1.2) node[anchor=east] {$x_{{\scriptscriptstyle 1,1}}$};
\filldraw [black] (2.3,2.8) node[anchor=west] {$u^{(2)}_{1}$};
\filldraw [black] (5.9,4.2) node[anchor=west] {$u^{(2)}_{2}$};
\filldraw [black] (4.1,2.9) node[anchor=south] {$x_{{\scriptscriptstyle 2,1}}$};
\filldraw [black] (5.3,4.2) node[anchor=east] {$x_{{\scriptscriptstyle 2,2}}$};
\filldraw [black] (6.33,3) node[anchor=east] {$x_{{\scriptscriptstyle 2,3}}$};
\filldraw [black] (7.2,2.8) node[anchor=west] {$u^{(1)}_{n}$};
\filldraw [black] (7.1,1) node[anchor=west] {$u^{(2)}_{n}$};
\filldraw [black] (6.1,2.5) node[anchor=west] {$x_{{\scriptscriptstyle n,1}}$};
\filldraw [black] (7.6,2) node[anchor=north] {$x_{{\scriptscriptstyle n,2}}$};
\filldraw [black] (6.4,1.2) node[anchor=east] {$x_{{\scriptscriptstyle n,3}}$};
\end{tikzpicture}
\caption{The graph $\mathcal{G}_n$}
\label{fig1}
\end{figure}

Let $n$ be a positive integer. We consider a non-bipartite graph $\mathfrak{g}_{r_{1},\dots,r_{m}}$, consisting of $n$ odd cycles that share a single common vertex (see Figure~\ref{fig: Main graph}).
To be precise, for an integer $m\geq 1$ and  each $j\in [m]=\{1,\dots,m\}$, we define $\mathfrak{g}_{r_1,\dots,r_m}$ to be the graph consisting of $r_j$ cycles of length $2j+1$, such that all $n=\sum\limits_{j=1}^{m}r_{j}$ odd cycles share a single common vertex.

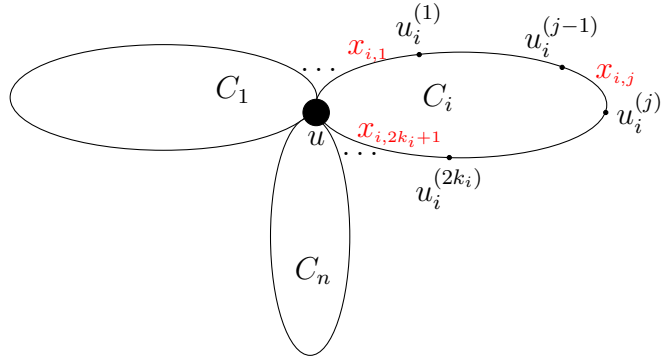
\begin{figure}[ht]
\centering
\begin{tikzpicture}
\filldraw[color=black!100, fill=white!5, thin](2.2,3.1) ellipse (58pt and 20pt);
\draw[color=black!100, fill=white!5,  thin](6.16,3) ellipse (55pt and 20pt);
\filldraw[color=black!100, fill=white!5, thin](4.15,1.25) ellipse (15pt and 45pt);
\filldraw [black] (4.23,2.9) circle (5pt);
\filldraw [black] (4.23,2.8)  node[anchor=north] {$u$};
\path (4.3,2.45) -- node[auto=false]{$\dots$} (4.3,4.5);
\path (4.9,1.8) -- node[auto=false]{$\dots$} (4.8,2.9);
\filldraw [black] (5.6,3.67) circle (0.8pt) node[anchor=south] {$u_{i}^{(1)}$};
\filldraw [black] (7.5,3.5) circle (0.8pt) node[anchor=south] {$u_{i}^{(j-1)}$};
\filldraw [black] (8.08,2.9) circle (0.8pt) node[anchor=west] {$u_{i}^{(j)}$};
\filldraw [black] (6,2.3) circle (0.8pt) node[anchor=north] {$u_{i}^{(2k_{i})}$};
\filldraw [black] (4.5,3.72) node[anchor=west] {\textcolor{red}{$x_{{\scriptscriptstyle i,1}}$}};
\filldraw [black] (7.8,3.4) node[anchor=west] {\textcolor{red}{$x_{{\scriptscriptstyle i,j}}$}};
\filldraw [black] (5.34,2.9) node[anchor=north] {\textcolor{red}{$x_{{\scriptscriptstyle i,2k_{i}+1}}$}};
\filldraw [black] (5.5,3.1) node[anchor=west] {$C_{i}$};
\filldraw [black] (3.8,0.8) node[anchor=west] {$C_{n}$};
\filldraw [black] (3.5,3.2) node[anchor=east] {$C_{1}$};
\end{tikzpicture}
\caption{The graph $\mathfrak{g}_{r_{1},\dots ,r_{m}}$ with $n=\sum\limits_{j=1}^{m}r_{j}$ }
\label{fig: Main graph}
\end{figure}

The main result of our paper is as follows:
\begin{theorem}\label{thm:main}

The $h$-polynomial of $\mathbb{K}[\mathfrak{g}_{r_1,\dots,r_m}]$ is given by
\begin{align}\label{eq:hvector}
    h(\mathbb{K}[\mathfrak{g}_{r_1,\dots,r_m}];t) = \prod_{j=1}^m(1+\dots +t^j)^{r_j}
   -t\prod_{j=1}^m(1+\dots + t^{j-1})^{r_j}. 
\end{align}
\end{theorem}
Using Theorem \ref{thm:main}, we are then able to characterize the graphs $\mathfrak{g}_{r_1,\dots ,r_m}$ such that $\mathbb{K}[\mathfrak{g}_{r_1,\dots ,r_m}]$ is almost Gorenstein.
\begin{theorem}\label{thm:almostGoren}
For the graph $\mathfrak{g}_{r_1,\dots ,r_m}$ with $n=\sum_{j=1}^{m}r_{j}$ and $N=\sum_{j=1}^{m} jr_{j}$, the edge ring $\mathbb{K}[\mathfrak{g}_{r_1, \dots, r_m}]$ is almost Gorenstein if and only if either
\begin{itemize}
    \item $n=1,2$; or 
    \item $n\geq3$ and every cycle in $\mathfrak{g}_{r_1, \dots, r_m}$ is a $3$-cycle, that is, $N=n$.
\end{itemize}
\end{theorem}

Our paper is organized as follows. In Section \ref{sec:prelims}, we recall the necessary background on graphs, edge rings, toric ideals, and simplicial complexes, and introduce the family of graphs $\mathfrak{g}_{r_{1},\dots ,r_{m}}$ we are concerned with. In Section \ref{sec:h-poly}, we prove our main result, Theorem \ref{thm:main}. Finally, in Section  \ref{sec: Gorenstein}, we describe when the edge ring of $\mathfrak{g}_{r_{1},\dots ,r_{m}}$ is almost Gorenstein.
\section{Background}\label{sec:prelims}

We now recall the necessary definitions and introduce some more notation to describe our results.
\subsection{Graph theory, toric ideals, edge rings, and simplicial complexes}

A {\it finite simple graph} (or a {\it graph}) $G=(V(G),E(G))$ consists 
of a non-empty finite set 
$V(G) = \{v_1,\ldots,v_n\}$, called the
vertices, and a finite set 
$E(G) = \{e_1,\ldots,e_q\} \subseteq \{\{u,v\}\mid 
u,v \in V(G), u\neq v\}$ of distinct unordered pairs 
of distinct elements of $V(G)$, called 
the edges.  We sometimes write $V$ (resp. $E$) for $V(G)$ (resp. $E(G)$) if $G$ is clear from the context.
A graph $H = (V(H),E(H))$ is said to be a {\it subgraph} of a graph $G$ if $V(H) \subseteq V(G)$ and  $E(H) \subseteq E(G)$. In this case, we say that $G$ contains $H$ and write $H \subseteq G$.

A {\it walk} of $G$ is a sequence of edges $(e_1, e_2, \dots, e_m)$, where each $e_i=\{v_i,w_i\}\in E$ and $w_i=v_{i+1}$ for each $i = 1, \dots, m - 1$. 
Equivalently, a walk is a sequence of vertices $(u_1, \dots,u_m,u_{m+1})$ such that 
$\{u_i,u_{i+1}\} \in E$ for all 
$i = 1,\dots,m$. Here, $m$ is referred to as the 
{\it length} of the walk. A walk is 
{\it even} if $m$ is even, and it is {\it closed} if $u_{m+1}=u_1$. Two vertices $u$ and $v$ are said to be {\it connected} if there is a walk between them. A graph $G$ is said to be {\it connected} if every two distinct vertices of $G$ are connected.

A {\it cycle} of a graph $G$ is a closed walk $(u_1, \dots,u_m,u_{m+1}=u_1)$ of vertices of $G$ (with $m \geq 3$) such that the only vertices in the walk that are not pairwise distinct are $u_1$ and $u_{m+1}$. A cycle of length $m$ is called an 
{\it $m$-cycle}.   

Let $G = (V,E)$ be a finite simple graph with edge set $E = \{e_1,\ldots,e_q\}$ ($q \geq 1$) and vertex set $V =\{v_1,\ldots,v_n\}$. Let $\mathbb{K}$ be a field. Define a $\mathbb{K}$-algebra homomorphism $\varphi:R = \mathbb{K}[e_1,\ldots,e_q] \rightarrow \mathbb{K}[v_1,\ldots,v_n]$ by $\varphi(e_i) = v_jv_k$ if $e_i = \{v_j,v_k\} \in E$.  The {\it toric ideal of $G$}, denoted $I_G$, is the kernel of $\varphi$. The \textit{edge ring} of $G$, denoted $\mathbb{K}[G]$, is the image of $\varphi$, and we have $\mathbb{K}[G]\cong R/I_G$ as rings.

Let $\Gamma=(e_{i_1},\ldots,e_{i_{2r}})$ be a closed even walk of a graph $G$. We define a binomial $f_\Gamma=f_\Gamma^{(+)}-f_\Gamma^{(-)}$, where 
\[f_\Gamma^{(+)}=\prod_{k=1}^rx_{i_{2k-1}} \;\text{ and }\;f_\Gamma^{(-)}=\prod_{k=1}^rx_{i_{2k}}.\]
We say that $\Gamma$ is a \textit{primitive} closed even walk if there is no even closed walk $\Gamma'$ of $G$ with $f_{\Gamma'} \neq f_\Gamma$ 
such that $f_{\Gamma'}^{(+)}$ divides $f_\Gamma^{(+)}$ and $f_{\Gamma'}^{(-)}$ divides $f_\Gamma^{(-)}$. The following result shows how we may obtain a generating set for $I_G$.
\begin{theorem}[{\cite{HHO}*{Lemma 5.10}}]\label{Thm:GB}
    The toric ideal $I_G$ of a graph $G$, then 
    \[I_G = \langle f_\Gamma ~|~ \mbox{$\Gamma$ is a primitive closed even walk of $G$}\rangle.\]
\end{theorem}
Moreover, this generating set is a Gr\"obner basis for $I_G$ with respect to any monomial order (e.g., see \cite{Villarreal-book}*{Proposition 10.1.10}).

We say that a connected graph $G$ satisfies the \textit{odd-cycle condition} if for any two odd cycles $A$ and $B$ of $G$ with $V(A) \cap V(B) = \emptyset$, there exists $i \in V(A)$ and $j\in  V(B)$ such that $e = \{i,j\} \in E(G)$. We then have the following combinatorial criterion for the edge ring $\mathbb{K}[G]$ to be \textit{normal}. (For an introduction to normality, see, e.g., \cite{BH}*{Section 6.1}.)

\begin{theorem}[{\cites{OH98, SVV}}]
Let $G$ be a connected graph. The edge ring $\mathbb{K}[G]$ is normal if and only if $G$ satisfies the odd-cycle condition. 
\end{theorem}
Note that if $G$ satisfies the odd-cycle condition, it follows from a result of Hochster~\cite{Hochster}*{Theorem 1} that $\mathbb{K}[G]$ is a Cohen--Macaulay ring. We refer the reader to  \cite{HHO}*{Section 5} and \cite{Villarreal-book}*{Section 10} for more details on edge rings and toric ideals of graphs.

A \textit{simplical complex} on a vertex set $X=\{x_1,\dots,x_n\}$ is a collection $\Delta$ of subsets of $X$ such that (i) if $F \in \Delta$ and $G \subseteq F$, then $G \in \Delta$, and (ii) for each $i=1,\dots,n$, $\{x_i\} \in \Delta$. An element $F \in \Delta$ is called a \textit{face}. A maximal element of $\Delta$ with respect to inclusion is called a \textit{facet}. The dimension of a face $F$ in $\Delta$ is defined as $\dim F=|F|-1$. The dimension of $\Delta$ is $\dim \Delta =\max\{\dim F \mid F \in \Delta\}$.

Given a squarefree monomial ideal $I$ in a polynomial ring $S=\mathbb{K}[x_1,\dots,x_n]$, the \textit{Stanley--Reisner complex} of $I$ is the simplical complex $ \Delta_I= \{A \subseteq \{x_1,\dots,x_n\} \mid m_A \not \in I\}$, where \[m_A=\prod_{x_i \in A} x_i.\] We call the quotient ring $S/I$ the \textit{Stanley-Reisner ring} of $\Delta_I$, and denote it by $\mathbb{K}[\Delta_I]$. For a graph $G$ with toric ideal $I_G$, we write  $\Delta_G$ for the Stanley-Reisner complex of $\text{in}_\prec(I_G)$ (the initial ideal of $I_G$ with respect to a monomial order $\prec$). (For an overview of Stanley--Reisner theory, see e.g., \cite{BH}*{Section 5}.)

\subsection{\texorpdfstring{Fundamental notions concerning the graph $\mathfrak{g}_{r_1, \dots, r_m}$}{The graph gn}}

In this subsection, we concentrate on our graph $\mathfrak{g}_{r_1, \dots, r_m}$ and its associated edge ring $\mathbb{K}[\mathfrak{g}_{r_1, \dots, r_m}]$.

Note that from here on, we will write $\mathfrak{g}_{n} \coloneqq \mathfrak{g}_{r_1, \dots , r_m}$ and $ \mathbb{K}[\mathfrak{g}_{n}] \coloneqq \mathbb{K}[\mathfrak{g}_{r_1, \dots , r_m}]$, where $n=\sum_{p=1}^{m}r_{p}$.
For each $1\leq i\leq n$, let $C_i$ be the $i^{\textrm{th}}$ cycle in $\mathfrak{g}_{n}$. The length of $C_i$ is $2 k_{i}+1$, where $k_{i} \in [m]$. 

Since all the odd cycles in $\mathfrak{g}_{n}$ share a common vertex, our graph $\mathfrak{g}_{n}$ satisfies the odd-cycle condition, and therefore, the edge ring $\mathbb{K}[\mathfrak{g}_{n}]$ is a normal Cohen--Macaulay homogeneous domain. 

The vertex set of $\mathfrak{g}_{n}$ is 
$\{u\}\cup \{u_{i}^{(j)}\colon 1\leq i\leq n, 1\leq j\leq 2k_{i}\}$ and
we label the edges of $\mathfrak{g}_{n}$ as $x_{i,j}$ such that  
\begin{itemize}
    \item for $1\leq i\leq n$, $x_{i,1}=\{u,u_{i}^{(1)}\}$ and $x_{i,2k_{i}+1}=\{u,u_{i}^{(2k_{i})}\}$, and   \vspace{0.1cm}
    \item for $1\leq i\leq n$ and $2\leq j\leq 2k_{i}$, $x_{i,j}=\{u_{i}^{(j-1)},u_{i}^{(j)}\}$. (For illustration, see Figure~\ref{fig: Main graph}.)
\end{itemize}

Let $R$ be the polynomial ring $\mathbb{K}[x_{i,j} \colon 1\leq i  \leq n, 1\leq j \leq 2k_{i}+1]$ where $ k_{i}\in [m]$, and we regard $I_{\mathfrak{g}_{n}}$ as an ideal of $R$. 

By \cite{HHO}*{Lemma 5.11}, every primitive even closed walk of the graph $\mathfrak{g}_{n}$ is given by 
\[(x_{i,1},x_{i,2},\dots x_{i,2k_{i}+1},x_{j,1},x_{j,2},\dots x_{j,2k_{j}+1}); \hspace{0.3cm}1\leq i < j \leq n, \text{ and } k_{i},k_{j}\in [m].\]
Therefore, the toric ideal $I_{\mathfrak{g}_{n}}$ is generated by the binomials:
\begin{align}\label{binom:G_n}
\prod\limits_{s=0}^{k_i}x_{i,2s+1}\prod\limits_{t=1}^{k_j}x_{j,2t} - \prod\limits_{s=1}^{k_i}x_{i,2s}\prod\limits_{t=0}^{k_j}x_{j,2t+1}; \hspace{0.3cm}1\leq i < j\leq n. 
\end{align}

Let $\prec$ be the graded lexicographic order on $R$ induced by the ordering of the variables 
\begin{align}\label{eq:lex} 
x_{1,1} \succ  x_{1,2} \succ\cdots \succ x_{1,2k_{1}+1} \succ\cdots \succ x_{n,1}  \succ x_{n,2} \succ \cdots  \succ x_{n,2k_{n}+1}.
\end{align}

From Theorem~\ref{Thm:GB} we have that the binomials in \eqref{binom:G_n} form a Gr\"{o}bner basis of $I_{\mathfrak{g}_{n}}$ with respect to the monomial order $\prec$. Hence the initial ideal $\ini_{\prec}( I_{\mathfrak{g}_{n}})$ of $I_{\mathfrak{g}_{n}}$ is generated by the squarefree monomials
\begin{align}\label{eq:initial}
\prod\limits_{s=0}^{k_i}x_{i,2s+1}\prod\limits_{t=1}^{k_j}x_{j,2t};\hspace{0.5cm} 1\leq i < j\leq n;\hspace{0.2cm} k_{i},k_{j}\in [m]. 
\end{align}

Since the monomial ideal $\ini_{\prec}( I_{\mathfrak{g}_{n}})$  is  squarefree, we can associate a Stanley--Reisner complex whose ideal coincides with the initial ideal generated by the monomials in \eqref{eq:initial}. 

\section{\texorpdfstring{The $h$-polynomial of $\mathbb{K}[\mathfrak{g}_{r_1, \dots , r_m}]$}{The h-polynomial of K[gn]}}\label{sec:h-poly}

In this section, we will explore specific constructions and key findings necessary to prove Theorem~\ref{thm:main}. 
Upon conclusion, we will utilize these insights to establish Theorem~\ref{thm:main}.

\subsection{\texorpdfstring{Stanley--Reisner complex $\Delta_{\mathfrak{g}_n}$}{Stanley--Reisner complex}}
Recall that $\Delta_{\mathfrak{g}_n}$ is the Stanley--Reisner complex of the initial ideal $\ini_{\prec}(I_{\mathfrak{g}_{n}})$.
Let $\mathcal{F}(\Delta_{\mathfrak{g}_n})$ be the set of all facets of $\Delta_{\mathfrak{g}_{n}}$.

For the graph $\mathfrak{g}_{n}$, let
\[\mathcal{O}_{i}\coloneqq \{x_{i,1},x_{i,3}, \dots, x_{i,2k_{i}+1}\};\hspace{0.2cm} \mathcal{E}_{i}\coloneqq \{x_{i,2},x_{i,4}, \dots, x_{i,2k_{i}}\},\] where $1\leq i\leq n$ and corresponding $k_{i}\in [m]$. 
Therefore by definition, any facet of $\Delta_{\mathfrak{g}}$ can be written as a maximal set that excludes the set $\mathcal{O}_{i}\cup \mathcal{E}_{j};\ 1\leq i < j\leq n.$ 
Note that $\mathcal{O}_{n}\cup \mathcal{E}_{1}\subset F$ for all $F\in \mathcal{F}(\Delta_{\mathfrak{g}_n})$. Hence, without loss of generality, let us express the facets without indicating any elements of $\mathcal{O}_{n}\cup \mathcal{E}_{1}$.

Now, let us get a more concrete representation for the facets in $\mathcal{F}(\Delta_{\mathfrak{g}_n})$. 
Note that if our graph $\mathfrak{g}_n$ consists of $n$ copies of $3$-cycles, then it is none other than the graph $\mathcal{G}_{n}$ illustrated in Figure~\ref{fig1}.
This family of graphs has been thoroughly examined in \cite{HN} and it was proved that any facet in this case is of the form 
\begin{align}\label{facets3-cyc}
\{z_{1},\dots ,z_{j-1}\}\cup \{x_{j,1},x_{j,3},\dots , x_{n-1,1},x_{n-1,3}\}\cup \{x_{2,2},\dots , x_{j,2}\},
\end{align}
where $z_{i}\in\{x_{i,1},x_{i,3}\}$ and $j=1,\dots , n.$

Using similar arguments from \cite{HN} concerning the explicit form of the facets, we can get a concrete representation for the facets of $\Delta_{\mathfrak{g}_n}$.
Let $F \in \mathcal{F}(\Delta_{\mathfrak{g}_{n}})$. We consider two cases.

\noindent 
$\textbf{\underline{Case 1:}}$ 
Let  $\mathcal{O}_{p}\subset F$ for some $1\leq p\leq n$.
By definition, facet $F$ does not contain the set $\mathcal{O}_{p}\cup \mathcal{E}_{q};\ 1\leq p < q\leq n$.
Therefore, $\mathcal{E}_{q}\not\subset F$ for all $q$ such that $ p< q\leq n$. Moreover, by maximality of $F$, every element of $\mathcal{E}_q$ excluding one will belong to $F$.

\noindent
$\textbf{\underline{Case 2:}}$ 
Some elements of $\mathcal{O}_{p}$ are contained in $F$, for any $1\leq p\leq n$. If there exists some $q$ with $p< q\leq n$ such that $\mathcal{E}_{q}\subset F$, then by the definition of facets, $\mathcal{O}_{p}\not\subset F$. In particular, by maximality, the facet $F$ contains every element of $\mathcal{O}_{p}$ except for one. On the other hand, if $\mathcal{E}_{q}\not\subset F$ for all $q$ with $p< q\leq n$, then by the maximality of $F$, we have $\mathcal{O}_{p}\subset F$.

Therefore, any facet $F\in\mathcal{F}(\Delta_{\mathfrak{g}_n})$ is of the form:
\begin{align}\label{eq:facets}
\bigcup\limits_{i=1}^{j-1}\zeta_{i}\cup \bigcup\limits_{i=2}^{j}\mathcal{E}_{i}\cup  \bigcup\limits_{i=j}^{n-1}\mathcal{O}_{i}\cup \bigcup\limits_{i=j+1}^{n}\omega_{i},
\end{align}
where $j=1,\dots , n$, such that $\zeta_{i}\in \binom{\mathcal{O}_{i}}{k_{i}}$, for all $ k_{i}\geq 1$ ($\zeta_i$ is a set containing any $k_{i}$ elements of $\mathcal{O}_{i}$), and 
$\omega_{i} \in \binom{\mathcal{E}_{i}}{k_{i}-1} $ for all $k_{i}\geq 2$. By default, let $\omega_i =\emptyset$ for any $k_{i}=1$.

\begin{remark} 
Let us express the facets in \eqref{facets3-cyc} using the generalized expression \eqref{eq:facets}. 
For the particular case of graph $\mathfrak{g}_{n}$ in Figure~\ref{fig1}, $k_{i}=1$ for all $1\leq i\leq n$ and thus $\bigcup_{i=j+1}^{n}\omega_{i}$ will not occur in the expression and 
\begin{itemize}
   \item $\{z_{1},\dots ,z_{j-1}\}$, where $z_{i}\in\{x_{i,1},x_{i,3}\}=\mathcal{O}_{i}$ corresponds to $\bigcup\limits_{i=1}^{j-1}\zeta_{i}$,
    \item $\{x_{j,1},x_{j,3},\dots , x_{n-1,1},x_{n-1,3}\}$ corresponds to $\bigcup\limits_{i=j}^{n-1}\mathcal{O}_{i}$ and 
    \item $\{x_{2,2},\dots , x_{j,2}\}$ corresponds to $\bigcup\limits_{i=2}^{j}\mathcal{E}_{i}$, where $\mathcal{E}_i =\{x_{i,2}\}$.
\end{itemize}
Therefore, for the family of graphs $\mathfrak{g}_{n}$ with $n$ copies of $3$-cycles, any facet is of the form: $\bigcup\limits_{i=1}^{j-1}\zeta_{i}\cup \bigcup\limits_{i=2}^{j}\mathcal{E}_{i}\cup  \bigcup\limits_{i=j}^{n-1}\mathcal{O}_{i}, \text{ where } j=1,\dots , n.$
\end{remark}

As a way to better comprehend the facets of $\Delta_{\mathfrak{g}_n}$, let's look at a running example.
\begin{example}
Let us consider the graph  $\mathfrak{g}_{3}\coloneqq \mathfrak{g}_{1,1,1} $ in Figure~\ref{figExample} with $k_1 =3$, $k_2 =2$, $k_3 =1$. 

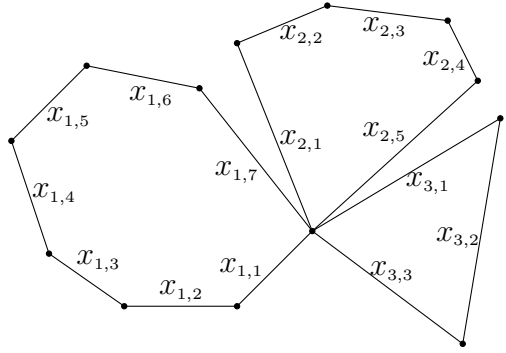
\begin{figure}[ht]
\centering
\begin{tikzpicture}
\draw[black, thin] (5,2) -- (4,1) -- (2.5,1) -- (1.5,1.7) -- (1,3.2) --  (2,4.2)-- (3.5,3.9) -- cycle;
\draw[black, thin] (5,2) -- (4,4.5) -- (5.2,5)-- (6.8,4.8) -- (7.2,4)-- cycle;
\draw[black, thin] (5,2) -- (7.5,3.5) -- (7,0.5)-- cycle;
\filldraw [black] (5,2) circle (1pt);
\filldraw [black] (4,1) circle (1pt);
\filldraw [black] (2.5,1) circle (1pt);
\filldraw [black] (1.5,1.7) circle (1pt);
\filldraw [black] (1,3.2) circle (1pt);
\filldraw [black] (2,4.2) circle (1pt);
\filldraw [black] (3.5,3.9) circle (1pt);
\filldraw [black] (4,4.5) circle (1pt);
\filldraw [black] (5.2,5) circle (1pt);
\filldraw [black] (6.8,4.8) circle (1pt);
\filldraw [black] (7.2,4) circle (1pt);
\filldraw [black] (7.5,3.5) circle (1pt);
\filldraw [black] (7,0.5) circle (1pt);
\filldraw [black] (4.5,1.5) node[anchor=east] {$x_{{\scriptscriptstyle 1,1}}$};
\filldraw [black] (3.7,1.2) node[anchor=east] {$x_{{\scriptscriptstyle 1,2}}$};
\filldraw [black] (1.57,2.2) node[anchor=south] {$x_{{\scriptscriptstyle 1,4}}$};
\filldraw [black] (2.6,1.6) node[anchor=east] {$x_{{\scriptscriptstyle 1,3}}$};
\filldraw [black] (2.2,3.5) node[anchor=east] {$x_{{\scriptscriptstyle 1,5}}$};
\filldraw [black] (3.3,3.78) node[anchor=east] {$x_{{\scriptscriptstyle 1,6}}$};
\filldraw [black] (4,2.5) node[anchor=south] {$x_{\scriptscriptstyle 1,7}$};
\filldraw [black] (5.3,3.2) node[anchor=east] {$x_{{\scriptscriptstyle 2,1}}$};
\filldraw [black] (5.3,4.6) node[anchor=east] {$x_{{\scriptscriptstyle 2,2}}$};
\filldraw [black] (6.4,4.7) node[anchor=east] {$x_{{\scriptscriptstyle 2,3}}$};
\filldraw [black] (7.2,4.25) node[anchor=east] {$x_{{\scriptscriptstyle 2,4}}$};
\filldraw [black] (6.4,3.3) node[anchor=east] {$x_{{\scriptscriptstyle 2,5}}$};
\filldraw [black] (6.1,2.64) node[anchor=west] {$x_{{\scriptscriptstyle 3,1}}$};
\filldraw [black] (6.94,2.2) node[anchor=north] {$x_{{\scriptscriptstyle 3,2}}$};
\filldraw [black] (6.5,1.45) node[anchor=east] {$x_{{\scriptscriptstyle 3,3}}$};
\end{tikzpicture}
\caption{The graph $\mathfrak{g}_{1,1,1}$ with $k_1 =3$, $k_2 =2$, $k_3 =1$}
\label{figExample}
\end{figure}

\noindent
For this graph, the initial ideal $\ini_{\prec}( I_{\mathfrak{g}_{3}})$ is 
\[\ini_{\prec}(I_{\mathfrak{g}_{3}}) = \langle
x_{\scriptscriptstyle{1,1}}x_{\scriptscriptstyle{1,3}}x_{\scriptscriptstyle{1,5}}x_{\scriptscriptstyle{1,7}}x_{\scriptscriptstyle{2,2}}x_{\scriptscriptstyle{2,4}}\ , \
x_{\scriptscriptstyle{1,1}}x_{\scriptscriptstyle{1,3}}x_{\scriptscriptstyle{1,5}}x_{\scriptscriptstyle{1,7}}x_{\scriptscriptstyle{3,2}}\ , \
x_{\scriptscriptstyle{2,1}}x_{\scriptscriptstyle{2,3}}x_{\scriptscriptstyle{2,5}}x_{\scriptscriptstyle{3,2}}
\rangle .\]
Note that since $k_3 = 1$, $\omega_3 =\emptyset$ in \eqref{eq:facets} for the graph $\mathfrak{g}_{3}$. As per \eqref{eq:facets}, the facets of $\Delta_{\mathfrak{g}_{3}}$ in Figure~\ref{figExample} is as follows:

\begin{equation*}
 \begin{aligned}
     & j=1 \ \ \ \ \underbrace{\{x_{\scriptscriptstyle{1,1}} , x_{\scriptscriptstyle{1,3}} , x_{\scriptscriptstyle{1,5}} , x_{\scriptscriptstyle{1,7}}\}}_{\mathcal{O}_1}\cup \underbrace{\{x_{\scriptscriptstyle{2,1}} , x_{\scriptscriptstyle{2,3}} , x_{\scriptscriptstyle{2,5}}\}}_{\mathcal{O}_2}\cup \underbrace{\binom{\{x_{\scriptscriptstyle{2,2}}, x_{\scriptscriptstyle{2,4}}\}}{1}}_{\omega_2}\\
     & j=2 \ \ \ \ \underbrace{\binom{\mathcal{O}_1}{3}}_{\zeta_1}\cup \underbrace{\{x_{\scriptscriptstyle{2,1}} , x_{\scriptscriptstyle{2,3}} , x_{\scriptscriptstyle{2,5}}\}}_{\mathcal{O}_2}\cup  \underbrace{\{x_{\scriptscriptstyle{2,2}}  , x_{\scriptscriptstyle{2,4}}\}}_{\mathcal{E}_2}\\     
     &j=3 \ \ \ \ \underbrace{\binom{\mathcal{O}_1}{3}}_{\zeta_1}\cup \underbrace{\binom{\mathcal{O}_2}{2}}_{\zeta_2}\cup \underbrace{\{x_{\scriptscriptstyle{2,2}}, x_{\scriptscriptstyle{2,4}}\}}_{\mathcal{E}_2}\cup
     \underbrace{\{x_{\scriptscriptstyle{3,2}}\}}_{\mathcal{E}_3}
 \end{aligned}
 \end{equation*} 
\end{example}

\begin{remark}
The Stanley--Reisner complex $\Delta_{\mathfrak{g}_{n}}$ is shellable and an extended version of the shelling order described in \cite{HN} can be used to get a shelling.
However, since the shellability is not used in the proofs of our main theorems, we omit the proof.   
\end{remark}

\subsection{Construction of graph \texorpdfstring{$\mathfrak{g}_{n}$}{gn} for the proof}\label{contruct}
Consider a graph $\mathfrak{g}_{n}^{\prime}$  with $n$ odd cycles $C_{i}^{\prime}$ such that length of each $C_{i}^{\prime}$ is $2k_{i}^{\prime}+1$, where $1\leq i\leq n$ and $k_{i}^{\prime}\in [m]$.
We construct the graph $\mathfrak{g}_{n}$ from $\mathfrak{g}_{n}^{\prime}$ by extending the odd cycle $C_{1}^{\prime}$ of $\mathfrak{g}_{n}^{\prime}$ by two edges while leaving the other odd cycles unchanged.
For illustration, see Figure \ref{fig:Con}.

\begin{figure}[ht]  
\centering 
\begin{tikzpicture}[scale=0.68]
\filldraw[color=black!100, fill=white!5, thin](2.2,3.1) ellipse (58pt and 20pt);
\draw[color=black!100, fill=white!5,  thin](6.16,3) ellipse (55pt and 20pt);
\filldraw[color=black!100, fill=white!5, thin](4.15,1.25) ellipse (15pt and 45pt);
\filldraw [black] (4.23,2.9) circle (5pt);
\filldraw [black] (2.4,3.8) circle (0.8pt);
\path (5.3,1.6) -- node[auto=false]{$\dots$} (4.8,2.9);
\filldraw [black] (2.2,2.39) circle (0.8pt);
\filldraw [black] (2.3,2) node[anchor=west] {\textcolor{red}{$x_{{\scriptscriptstyle  1,1}}$}};
\filldraw [black] (2.4,3.89) node[anchor=west] {\textcolor{red}{$x_{{\scriptscriptstyle  1,2k_{1}^{\prime}+1}}$}};
\filldraw [black] (5.5,3.1) node[anchor=west] {$C_{2}^{\prime}$};
\filldraw [black] (3.6,0.8) node[anchor=west] {$C_{n}^{\prime}$};
\filldraw [black] (3.5,3.2) node[anchor=east] {$C_{1}^{\prime}$};
\end{tikzpicture}
\begin{tikzpicture} [scale=0.90]
\filldraw[color=black!100, fill=white!5, thin](2.2,3.1) ellipse (58pt and 20pt);
\draw[color=black!100, fill=white!5,  thin](6.16,3) ellipse (55pt and 20pt);
\filldraw[color=black!100, fill=white!5, thin](4.15,1.25) ellipse (15pt and 45pt);
\filldraw [black] (4.23,2.9) circle (5pt);
\filldraw [black] (2.4,2.4) circle (0.8pt);
\filldraw [black] (0.7,3.57) circle (0.8pt);
\filldraw [black] (2.4,3.79) circle (0.8pt);
\filldraw [black] (3.3,3.68) circle (0.8pt);
\path (4.9,1.8) -- node[auto=false]{$\dots$} (4.8,2.9);
\filldraw [black] (2.83,2.3) node[anchor=west] {\textcolor{red}{$x_{{\scriptscriptstyle  1,1}}$}};
\filldraw [black] (0.6,3.98) node[anchor=west] {\textcolor{red}{$x_{{\scriptscriptstyle  1,2k_{1}^{\prime}+1}}$}};
\filldraw [black] (2.6,3.96) node[anchor=west] {\textcolor{red}{$y$}};
\filldraw [black] (3.5,3.7) node[anchor=west] {\textcolor{red}{$x$}};
\filldraw [black] (5.5,3.1) node[anchor=west] {$C_{2}^{\prime}$};
\filldraw [black] (3.8,0.8) node[anchor=west] {$C_{n}^{\prime}$};
\filldraw [black] (3.5,3.2) node[anchor=east] {$C_{1}$};
\filldraw [black] (1,2) node[anchor=east] {$\rightsquigarrow$};
\end{tikzpicture} 
\caption{Left: Graph $\mathfrak{g}_{n}^{\prime}$. Right: Graph $\mathfrak{g}_{n}$ } 
\label{fig:Con}  
\end{figure}
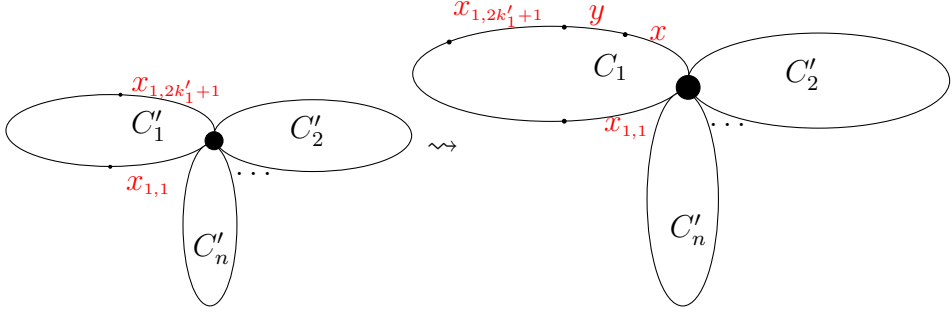  

For the graph $\mathfrak{g}_{n}^{\prime}$, we denote $\mathcal{O}_{i}^{\prime}\coloneqq \{x_{i,1},x_{i,3}, \dots, x_{i,2k_{i}^{\prime}+1}\}$; $\mathcal{E}_{i}^{\prime}\coloneqq \{x_{i,2},x_{i,4}, \dots, x_{i,2k_{i}^{\prime}}\}$ for all $1\leq i\leq n$, where $k_{i}^{\prime}\in [m]$. Accordingly, we denote $\zeta_{i}^{\prime}$ and $\omega_{i}^{\prime}$.
Recall that for all $1\leq i\leq n$, the cycles $C_i$ in $\mathfrak{g}_{n}$ are of length $2k_{i}+1$, where $k_{i}\in [m]$, and by construction we have  
\begin{itemize}
    \item $C_{i}=C_{i}^{\prime}$ for all $i\neq 1$,
    \item $k_{1}=  k_{1}^{\prime}+1$ and $k_{i} =  k_{i}^{\prime}$, for all $i\neq 1$.
    \item $x\coloneqq x_{1,2k_{1}+1}$ and $y\coloneqq x_{1,2k_{1}}$.
\end{itemize}

Let $\mathfrak{g}_{n}^{\prime\prime}\coloneqq \mathfrak{g}_{n}^{\prime}\backslash C_{1}^{\prime}$, the subgraph of $\mathfrak{g}_{n}^{\prime}$ which contains all the $C_{i}^{\prime}$ except $C_{1}^{\prime}$. 
Figure \ref{fig:G''} demonstrates the graph $\mathfrak{g}_{n}^{\prime\prime}$.
As per the construction, the cycles in $\mathfrak{g}_{n}^{\prime\prime}$ are $C_{i}^{\prime\prime}\coloneqq C_{i+1}^{\prime}$, with $k_{i}^{\prime\prime}\in [m]$ and $k_{i}^{\prime\prime}=k_{i+1}^{\prime}$, for all $1\leq i \leq n-1$. 
Moreover, $\mathcal{O}_{i}^{\prime\prime}$, $\mathcal{E}_{i}^{\prime\prime}$, $\zeta_{i}^{\prime\prime}$ and $\omega_{i}^{\prime\prime}$ are the corresponding notations.

\begin{figure}[ht]
\centering
\begin{tikzpicture}
\draw[color=black!100, fill=white!5,  thin](6.16,3) ellipse (55pt and 20pt);
\filldraw[color=black!100, fill=white!5, thin](4.15,1.25) ellipse (15pt and 45pt);
\filldraw [black] (4.23,2.9) circle (5pt);
\path (4.9,1.8) -- node[auto=false]{$\dots$} (4.8,2.9);
\filldraw [black] (5.5,3.1) node[anchor=west] {$C_{2}^{\prime}$};
\filldraw [black] (3.8,0.8) node[anchor=west] {$C_{n}^{\prime}$};
\end{tikzpicture}
\caption{The graph $\mathfrak{g}_{n}^{\prime\prime}$ }
\label{fig:G''}
\end{figure}
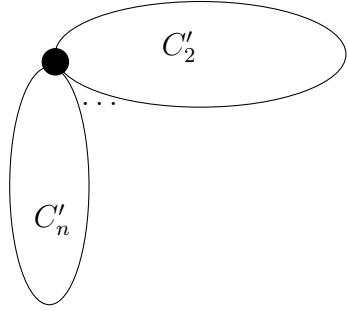

\subsection{Towards the Proof of Theorem~\ref{thm:main}}\label{subsec:proofrelated}
In this subsection, we re-examine the Stanley--Reisner complex $\Delta_{\mathfrak{g}_{n}}$ in light of our special construction of the graph $\mathfrak{g}_{n}$. 
Additionally, we explore feasible lattice points for any $\mathfrak{g}_n$. 
Using these insights and with an inductive approach, we prove Theorem~\ref{thm:main}.

For a simplicial complex $\Delta$ and any nonempty set $X$, let $\Delta * X$ represent the simplicial complex whose facets are of the form $F\cup X$, where $F\in \mathcal{F}(\Delta)$. 
According to our construction of the graph $\mathfrak{g}_{n}$, we have the following lemma.

\begin{lemma}\label{lemma:main}
For the graph $\mathfrak{g}_{n}$, we can express 
$\Delta_{\mathfrak{g}_{n}} = \Delta_{\mathfrak{g}_{n}^{\prime}}*\{x\}\cup \Delta^{\prime\prime}*\mathcal{O}_{1}^{\prime},$
where $\Delta^{\prime\prime}= \Delta_{\mathfrak{g}_{n}^{\prime\prime}} *  \mathcal{E}_{2}^{\prime}$.
\end{lemma}

\begin{proof}
With $k_{i}\in [m]$, the facets of $\Delta_{\mathfrak{g}_{n}}$ are of the form \eqref{eq:facets}.
By construction, we have $\mathcal{O}_{1}= \mathcal{O}_{1}^{\prime}\cup \{x\}$ and for all  $2\leq i\leq n$, $\mathcal{O}_{i}= \mathcal{O}_{i}^{\prime}$, $\mathcal{E}_{i}= \mathcal{E}_{i}^{\prime}$ and $\omega_{i}= \omega_{i}^{\prime}$.
Using this fact along with the general form of facets given in \eqref{eq:facets}, we have that any facet $F\in \mathcal{F}(\Delta_{\mathfrak{g}_{n}})$ has one of the following forms:
\begin{subequations}
  \begin{equation}
    \label{eq:a}
      \{x\}\cup\bigcup\limits_{i=1}^{j-1}\zeta_{i}^{\prime}\cup \bigcup\limits_{i=2}^{j}\mathcal{E}_{i}^{\prime}\cup  \bigcup\limits_{i=j}^{n-1}\mathcal{O}_{i}^{\prime}\cup \bigcup\limits_{i=j+1}^{n}\omega_{i}^{\prime}, \textrm{ where } 1\leq j \leq n,  \textrm{ or }
  \end{equation}
  \begin{equation}
    \label{eq:b}
    \mathcal{O}_{1}^{\prime}\cup\bigcup\limits_{i=2}^{j-1}\zeta_{i}^{\prime}\cup \bigcup\limits_{i=2}^{j}\mathcal{E}_{i}^{\prime}\cup  \bigcup\limits_{i=j}^{n-1}\mathcal{O}_{i}^{\prime}\cup \bigcup\limits_{i=j+1}^{n}\omega_{i}^{\prime}, \textrm{ where } 2\leq j \leq n.
  \end{equation}
\end{subequations}

Now, let us focus on the graph $\mathfrak{g}_{n}^{\prime\prime}$. 
Using \eqref{eq:facets} for $\mathfrak{g}_{n}^{\prime\prime}$, the facets of $\Delta_{\mathfrak{g}_{n}^{\prime\prime}}$ are of the form 
$\bigcup\limits_{i=1}^{j-1}\zeta_{i}^{\prime\prime}\cup \bigcup\limits_{i=2}^{j}\mathcal{E}_{i}^{\prime\prime}\cup  \bigcup\limits_{i=j}^{n-2}\mathcal{O}_{i}^{\prime\prime}\cup \bigcup\limits_{i=j+1}^{n-1}\omega_{i}^{\prime\prime},$ where $1\leq j\leq n-1$.
We know that the cycles in $\mathfrak{g}_{n}^{\prime\prime}$ are given by $C_{j}^{\prime\prime}\coloneqq C_{j+1}^{\prime}$, for all $1\leq j \leq n-1$. Thus, we have $\mathcal{E}_{j}^{\prime\prime}= \mathcal{E}_{j+1}^{\prime}$ and $\mathcal{O}_{j}^{\prime\prime}= \mathcal{O}_{j+1}^{\prime}$, for all $1\leq j \leq n-1$. 
Therefore, the facets of $\Delta_{\mathfrak{g}_{n}^{\prime\prime}}$ are of the form
$\bigcup\limits_{i=2}^{j}\zeta_{i}^{\prime}\cup \bigcup\limits_{i=3}^{j+1}\mathcal{E}_{i}^{\prime}\cup  \bigcup\limits_{i=j+1}^{n-1}\mathcal{O}_{i}^{\prime}\cup \bigcup\limits_{i=j+2}^{n}\omega_{i}^{\prime},$ where $1\leq j\leq n-1$.
Now, let us rewrite this expression such that the facets of $\Delta_{\mathfrak{g}_{n}^{\prime\prime}}$ are of the form 
\begin{align}\label{eq:facets''}
\bigcup\limits_{i=2}^{j-1}\zeta_{i}^{\prime}\cup \bigcup\limits_{i=3}^{j}\mathcal{E}_{i}^{\prime}\cup  \bigcup\limits_{i=j}^{n-1}\mathcal{O}_{i}^{\prime}\cup \bigcup\limits_{i=j+1}^{n}\omega_{i}^{\prime}, \textrm{ where } 2 \leq j \leq n.
\end{align}

Let $\Delta^{\prime\prime}$ be the simplicial complex whose facets are  
\begin{align}\label{eq:facetsDelta}
\bigcup\limits_{i=2}^{j-1}\zeta_{i}^{\prime}\cup \bigcup\limits_{i=2}^{j}\mathcal{E}_{i}^{\prime}\cup  \bigcup\limits_{i=j}^{n-1}\mathcal{O}_{i}^{\prime}\cup \bigcup\limits_{i=j+1}^{n}\omega_{i}^{\prime}, \textrm{ where } 2 \leq j \leq n.
\end{align}
From \eqref{eq:facets''} and \eqref{eq:facetsDelta}, it is evident that any facet of $\Delta^{\prime\prime}$ is of the form $F^{\prime\prime}\cup\mathcal{E}_{2}^{\prime}$, where $F^{\prime\prime}\in\mathcal{F}(\Delta_{\mathfrak{g}_{n}^{\prime\prime}})$. Therefore, we can express $\Delta^{\prime\prime}$ as
\begin{align}\label{eq:Delta''}
\Delta^{\prime\prime}= \Delta_{\mathfrak{g}_{n}^{\prime\prime}} *  \mathcal{E}_{2}^{\prime}.
\end{align}

Note that any facet $F^{\prime}\in\mathcal{F}(\Delta_{\mathfrak{g}_{n}^{\prime}})$ is of the form $\bigcup\limits_{i=1}^{j-1}\zeta_{i}^{\prime}\cup \bigcup\limits_{i=2}^{j}\mathcal{E}_{i}^{\prime}\cup  \bigcup\limits_{i=j}^{n-1}\mathcal{O}_{i}^{\prime}\cup \bigcup\limits_{i=j+1}^{n}\omega_{i}^{\prime},$ where $1\leq j \leq n$. 
Consequently, any facet of $\Delta_{\mathfrak{g}_{n}}$ of the form \eqref{eq:a} corresponds to $F^{\prime} \cup \{x\}$, where $F^{\prime}\in\mathcal{F}(\Delta_{\mathfrak{g}_{n}^{\prime}})$. 
It is thus possible to state that the facets in \eqref{eq:a} correspond to facets of $\Delta_{\mathfrak{g}_{n}^{\prime}}*\{x\}$.
By comparing \eqref{eq:b} and \eqref{eq:facetsDelta}, we observe that the collection of facets in \eqref{eq:b} matches $\Delta^{\prime\prime} *\mathcal{O}_1^{\prime}$.
Therefore, we have
$\Delta_{\mathfrak{g}_{n}}=\Delta_{\mathfrak{g}_{n}^{\prime}}*\{x\}\cup \Delta^{\prime\prime}*\mathcal{O}_{1}^{\prime}.$
\end{proof}

By examining the two forms of facets in $\mathcal{F}(\Delta_{\mathfrak{g}_{n}})$, it is evident that the intersection of any facet in \eqref{eq:a} with any other facet of \eqref{eq:b} is equal to some facet of $\Delta_{\mathfrak{g}_{n}^{\prime}}$. 
Hence we have 
\[\Delta_{\mathfrak{g}_{n}^{\prime}}*\{x\}\cap \Delta^{\prime\prime}*\mathcal{O}_{1}^{\prime}= \Delta_{\mathfrak{g}_{n}^{\prime}}.\]

Observe that $\mathcal{E}_{2}^{\prime}$ in \eqref{eq:Delta''} is just a simplex and therefore the $h$-polynomial of the Stanley--Reisner ring $\mathbb{K}[\Delta^{\prime\prime}]$ equals $h(\mathbb{K}[\Delta_{\mathfrak{g}_{n}^{\prime\prime}}]; t)$.
Lemma~\ref{lemma:main} leads us to the observation that all the simplicial complexes $\Delta_{\mathfrak{g}_{n}},\  \Delta_{\mathfrak{g}_{n}^{\prime}}*\{x\}$ and $\Delta^{\prime\prime}*\mathcal{O}_{1}^{\prime}$ have same dimension.
Furthermore, we have $\Delta_{\mathfrak{g}_{n}^{\prime}}*\{x\}\cap \Delta^{\prime\prime}*\mathcal{O}_{1}^{\prime}= \Delta_{\mathfrak{g}_{n}^{\prime}}$ and $\dim \Delta_{\mathfrak{g}_{n}^{\prime}}= \dim\Delta_{\mathfrak{g}_{n}}-1$.
Let us assume that $\dim \mathbb{K}[\Delta_{\mathfrak{g}_{n}}]=d$.

Since the Hilbert series of $R/I$ coincides with that of $R/\ini_{\prec}(I)$ in general (see, e.g., \cite{HHO}*{Proposition 2.6}), we conclude the $h$-polynomial of the Stanley--Reisner ring $\mathbb{K}[\Delta_{\mathfrak{g}_{n}}]$ gives the desired $h$-polynomial of corresponding edge ring $\mathbb{K}[\mathfrak{g}_{n}]$. 
Therefore by applying the inclusion-exclusion principle to the Hilbert series of $\mathbb{K}[\Delta_{\mathfrak{g}_{n}}]$ in Lemma~\ref{lemma:main}, we have
\[\frac{h(\mathbb{K}[\mathfrak{g}_{n}];t)}{(1-t)^{d}} = \frac{h(\mathbb{K}[\mathfrak{g}_{n}^{\prime}];t)}{(1-t)^{d}} + \frac{h(\mathbb{K}[\mathfrak{g}_{n}^{\prime\prime}];t)}{(1-t)^{d}}  - \frac{h(\mathbb{K}[\mathfrak{g}_{n}^{\prime}];t)}{(1-t)^{d-1}} .\]
As a result, the $h$-polynomials are as follows:
\begin{equation}\label{eq:hpoly}
\begin{split}
h(\mathbb{K}[\mathfrak{g}_{n}];t)&=h(\mathbb{K}[\mathfrak{g}_{n}^{\prime}];t) + h(\mathbb{K}[\mathfrak{g}_{n}^{\prime\prime}];t) - (1-t) h(\mathbb{K}[\mathfrak{g}_{n}^{\prime}];t) \\ 
 &= t\ h(\mathbb{K}[\mathfrak{g}_{n}^{\prime}];t) + h(\mathbb{K}[\mathfrak{g}_{n}^{\prime\prime}];t).
\end{split}
\end{equation}


For any graph $\mathfrak{g}_{n}=\mathfrak{g}_{r_1,\dots ,r_m}$ ($m\geq 1$), we can assign the lattice points $(n,N)$ where $n=\sum\limits_{j\geq1}r_{j}$ (total number of odd cycles) and $N=\sum\limits_{j\geq1} jr_{j}$, such that $N\geq  n \geq 1 $.
The feasible lattice points corresponding to any $\mathfrak{g}_{n}$ are shown in Figure~\ref{range}.

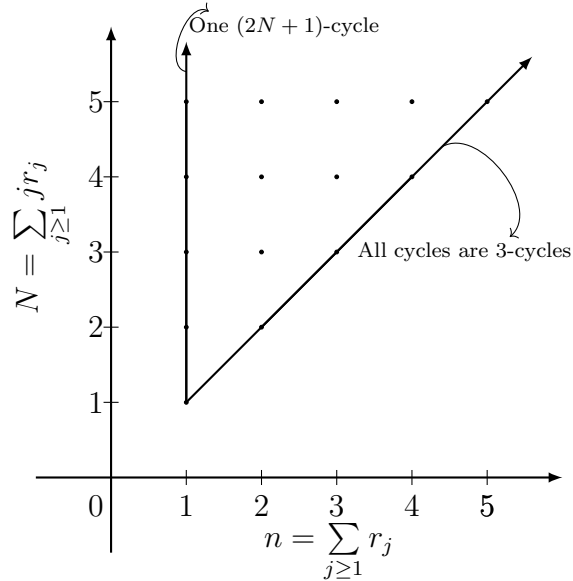
\begin{figure}[ht]  
\centering 
\begin{tikzpicture}
    [
        dot/.style={circle,draw=black, fill,inner sep=0.5pt},
    ]

\foreach \x in {1,...,5}{
    \node[dot] at (\x,5){ };
}
\foreach \x in {1,...,4}{
    \node[dot] at (\x,4){ };
}
\foreach \x in {1,...,3}{
    \node[dot] at (\x,3){ };
}
\foreach \x in {1,2}{
    \node[dot] at (\x,2){ };
}
\node[dot] at (1,1){ };

\foreach \x in {1,1}
    \draw[thick] (\x,1) -- (\x,5);
\draw[->,thick,-latex] (1,5) -- (1,5.8);
\foreach \x in {1,2}
    \draw[thick] (\x,\x) -- (\x+1,\x+1);
\foreach \x in {2,3}
    \draw[thick] (\x,\x) -- (\x+1,\x+1);
\foreach \x in {3,4}
    \draw[thick] (\x,\x) -- (\x+1,\x+1);
\draw[->,thick,-latex] (5,5) -- (5.6,5.6);    
\draw [->] (4.4,4.4) to [in=40,out=30] (5.3,3.2);
\node[xshift=-6mm] at (5.3,3) {{\tiny All cycles are $3$-cycles}}; 
\draw [->] (1,5.4) to [in=150,out=170] (1.3,6.2);
\node[xshift=-1mm] at (2.4,6) {{\tiny One $(2N+1)$-cycle}}; 

\foreach \x in {1,...,5}
    \draw (\x,.1) -- node[below,yshift=-1mm] {\x} (\x,-.1);
\node[below,yshift=-1mm] at (5,0) {5};
\node[below,xshift=-2mm,yshift=-1mm] at (0,0) {0};

\foreach \y in {1,...,5}
    \draw (.1,\y) -- node[xshift=-2mm] {\y} (-.1,\y);  
\node[xshift=-6mm] at (-0.3,3.2) { \rotatebox{90}{$N=\sum\limits_{j\geq 1}jr_{j}$}}; 
\node[xshift=-6mm] at (3.5,-1) { $n=\sum\limits_{j\geq 1}r_{j}$}; 
\draw[->,thick,-latex] (0,-1) -- (0,6);
\draw[->,thick,-latex] (-1,0) -- (6,0);
\end{tikzpicture}
\caption{Feasible lattice points corresponding to the graph $\mathfrak{g}_{r_1, \dots , r_m}$}
\label{range}
\end{figure}

Any graph $\mathfrak{g}_{n}$ with an arbitrary lattice point $(n, N)$ can be constructed from a base graph with $n$ copies of $3$-cycles, i.e., with lattice point $(n,n)$, by a similar method of construction as discussed in Section~\ref{contruct}.
Thus, we demonstrate our main theorem using an inductive approach over $N$ in the feasible lattice points $(n, N)$ with $N\geq n \geq 1$.


\begin{proof}[\textbf{Proof of Theorem \ref{thm:main}}]
As earlier, let $\mathfrak{g}_{n} \coloneqq \mathfrak{g}_{r_1, \dots , r_m}$ with each  $i^{\textrm{th}}$ cycle of length $2 k_{i}+1$, where $k_{i} \in [m]$.

For the base case with $N\geq n=1$, i.e., lattice points $(1, N)$, the edge ring is isomorphic to a polynomial ring in $2N+1$ (length of the single odd cycle) variables and the $h$-polynomial has to be $1$. 
Therefore, \eqref{eq:hvector} stands true.

As we can see in Figure~\ref{range}, the graph with $n$ copies of $3$-cycles having a unique common point, or $N=n\geq 1$, is another base case. In \cite{HN}, this case has been examined. 
Thus, \cite{HN}*{Theorem 1.1} implies that our main theorem is true for the base case $N=n\geq 1$.

Let us assume that the $h$-polynomial is given by \eqref{eq:hvector}, for any graph of our concern with corresponding lattice point $(n,N)$.

As per the discussion in Section~\ref{contruct}, we construct our graph $\mathfrak{g}_{n}$ from $\mathfrak{g}_{n}^{\prime}$ by altering one of the odd cycles of $\mathfrak{g}_{n}^{\prime}$.
Let $\mathfrak{g}_{n}^{\prime}$ be a graph with lattice point $(n,N)$ such that $n=\sum\limits_{j\geq1}r_{j}^{\prime}$ and $N=\sum\limits_{j\geq1} jr_{j}^{\prime}$.
By construction, the graph $\mathfrak{g}_{n}^{\prime\prime}$ has lattice point $(n-1,\widetilde{N})$, where $n-1=\sum\limits_{j\geq1}r_{j}^{\prime\prime}$ and $\widetilde{N}=\sum\limits_{j\geq1}jr_{j}^{\prime\prime} = \sum\limits_{j\geq1}jr_{j}^{\prime}-k_{{\scriptscriptstyle 1}}^{\prime}<N$.
Therefore by the induction hypothesis, the formula \eqref{eq:hvector} holds for both $\mathfrak{g}_{n}^{\prime}$ and $\mathfrak{g}_{n}^{\prime\prime}$.
Now, as per the construction, graph $\mathfrak{g}_{n}$ has a total of $n=\sum\limits_{j\geq1}r_{j}$ odd cycles such that 
\begin{align*}\sum\limits_{j\geq1} jr_{j} &=  k_{1}(r_{{\scriptscriptstyle k_{{\scriptscriptstyle 1}}}}^{\prime}+1) +  (k_{1}-1)(r_{{\scriptscriptstyle k_{{\scriptscriptstyle 1}}-1}}^{\prime}-1) + \sum\limits_{\substack{j\geq1 \\ j\neq k_{{\scriptscriptstyle 1}},k_{{\scriptscriptstyle 1}}-1}} jr_{j}^{\prime}\\
  &= \sum\limits_{j\geq1} jr_{j}^{\prime} + 1 = N+1.
\end{align*}
Therefore, the lattice point corresponding to $\mathfrak{g}_{n}$ is $(n,N+1)$.

From \eqref{eq:hpoly}, we have
\[h(\mathbb{K}[\mathfrak{g}_{n}];t) = t\ h(\mathbb{K}[\mathfrak{g}_{n}^{\prime}];t)+ h(\mathbb{K}[\mathfrak{g}_{n}^{\prime\prime}];t).\]
By the induction hypothesis, we apply \eqref{eq:hvector} to  $h(\mathbb{K}[\mathfrak{g}_{n}^{\prime}];t)$ and $h(\mathbb{K}[\mathfrak{g}_{n}^{\prime\prime}];t) $, and we have
\begin{align*}
  h(\mathbb{K}[\mathfrak{g}_{n}];t)&=  t\prod_{j\geq 1}(1+\dots +t^j)^{r_{j}^{\prime}}
   -t^2\prod_{j\geq 1}(1+\dots + t^{j-1})^{r_{j}^{\prime}}\\
  & \hspace{0.4cm} + \prod_{j\geq 1}(1+\dots +t^j)^{r_{j}^{\prime\prime}}
  -t\prod_{j\geq 1}(1+\dots + t^{j-1})^{r_{j}^{\prime\prime}}, 
\end{align*}
where $k\coloneqq k_{{\scriptscriptstyle 1}}^{\prime}$;  $r_{j}^{\prime}  =r_{j}^{\prime\prime}$ for $j\neq k$ and $r_{k}^{\prime} =r_{k}^{\prime\prime} +1$ (by construction). 
Let 
\begin{equation*}
    r_{j}\coloneqq
    \begin{cases}
    r_{j}^{\prime\prime} + 1  & \text{ for } j=  k+1, \\
    r_{j}^{\prime\prime}  & \textrm{ otherwise. }
    \end{cases}
\end{equation*}  
Then we have

\begin{align*}
   h(\mathbb{K}[\mathfrak{g}_{n}];t) &= t(1+\dots +t^{k})\prod_{j\geq 1}(1+\dots +t^j)^{r_{j}^{\prime\prime}}\\
   & \hspace{0.3cm}-t^{2}(1+\dots +t^{k-1})\prod_{j\geq 1}(1+\dots + t^{j-1})^{r_{j}^{\prime\prime}}\\
   & \hspace{0.3cm} + \prod_{j\geq 1}(1+\dots +t^j)^{r_{j}^{\prime\prime}} -t\prod_{j\geq 1}(1+\dots + t^{j-1})^{r_{j}^{\prime\prime}}\\
   & = (1+\dots +t^{k+1})\prod_{j\geq 1}(1+\dots +t^j)^{r_{j}^{\prime\prime}}\\
   & \hspace{0.3cm} -t(1+\dots +t^{k})\prod_{j\geq 1}(1+\dots + t^{j-1})^{r_{j}^{\prime\prime}}\\
   & = \prod_{j\geq 1}(1+\dots +t^j)^{r_{j}}
   -t\prod_{j\geq 1}(1+\dots + t^{j-1})^{r_{j}}.
\end{align*}

\noindent
This concludes our proof.
\end{proof}

\section{\texorpdfstring{On almost Gorensteinness of $\mathbb{K}[\mathfrak{g}_{r_1, \dots, r_m}]$}{On almost Gorensteinness of K[gn]}} \label{sec: Gorenstein}

Let $R$ be a normal Cohen–Macaulay homogeneous domain and $h(R)=(h_0,h_1,\ldots,h_s)$ be its $h$-vector. From \cite{S78}*{Theorem 4.4}, we know that $R$ is Gorenstein if and only if $h(R)$ is symmetric, i.e.,
$h_i = h_{s-i}$ for $i = 0, 1,\dots,s$. 

The notion of almost Gorenstein homogeneous rings was given by Goto--Takahashi--Taniguchi in \cite{GTT}, as a new class of graded rings between Cohen--Macaulay rings and Gorenstein rings. 
For further studies on almost Gorenstein homogeneous rings, see, e.g., \cites{H, HM2, MM}. Note that if $R$ is Gorenstein, then it is almost Gorenstein. 
We now characterize the almost Gorensteinness of our edge ring $\mathbb{K}[\mathfrak{g}_{r_1, \dots, r_m}]$.

From \cite{H}, let us recall a necessary and sufficient condition for a homogeneous domain to be almost Gorenstein. 
\begin{proposition}[{\cite{H}*{Corollary 2.7}}]\label{prop:alm}
Let $R$ be a Cohen--Macaulay homogeneous domain of dimension $d$ over a field $\mathbb{K}$ and let $h(R)=(h_0,h_1,\ldots,h_s)$ be its $h$-vector. 
Then $R$ is almost Gorenstein if and only if the following equality holds: 
\begin{align*}
\typ{(R)}-1=\sum_{i=0}^{s}((h_s+\cdots+h_{s-i})-(h_0+\cdots+h_i)) =:\tilde{e}(R), 
\end{align*}
where $\typ{(R)}$ denotes the Cohen--Macaulay type of $R$.
\end{proposition}

It is well known that $\typ{(R)}$ coincides with the number of elements in the minimal system of generators of the canonical module of $R$. In the case of $\mathfrak{g}_{r_1,\dots,r_m}$, we have the following. 

\begin{proposition}[{\cite{compact}*{Theorem 5.7}}]\label{prop:CMtype}
For any graph $\mathfrak{g}_{r_1,\dots,r_m}$ consisting of $n=\sum_{j=1}^{m}r_{j}$ cycles,  $\typ{(\mathbb{K}[\mathfrak{g}_{r_1, \dots, r_m}])} = n-1.$
\end{proposition}

Now we are in the position to provide a proof for the second main theorem.

\begin{proof}[\textbf{Proof of Theorem~\ref{thm:almostGoren}}]
As before, let $\mathfrak{g}_{n} \coloneqq \mathfrak{g}_{r_1, \dots, r_m}$. 
We know that for $n=1$, the edge ring $\mathbb{K}[\mathfrak{g}_{n}]$ is isomorphic to a polynomial ring in $2N+1$ variables.  
For the case $n=2$, the toric ideal $I_{\mathfrak{g}_{n}}$ has only one generator and therefore the corresponding edge ring is a hypersurface.
Thus, the edge ring $\mathbb{K}[\mathfrak{g}_{n}]$ is Gorenstein for both $n=1,2$; which implies that $\mathbb{K}[\mathfrak{g}_{n}]$ is almost Gorenstein.

Now, let $n\geq 3$.
From Proposition~\ref{prop:alm} and Proposition~\ref{prop:CMtype}, we have that $\mathbb{K}[\mathfrak{g}_{n}]$ is almost Gorenstein if and only if $\tilde{e}(\mathbb{K}[\mathfrak{g}_{n}]) = n-2$. 
Let $h(\mathbb{K}[\mathfrak{g}_{n}])= (h_0,h_1,\ldots,h_s)$ be the $h$-vector of $\mathbb{K}[\mathfrak{g}_{n}]$ and let $h_{i}^{\prime} = (h_s+\cdots+h_{s-i})-(h_0+\cdots+h_i)$, for $i=0,1,\dots,s$. 
From the proof of \cite{H}*{Proposition 2.4}, we have 
\begin{equation}\label{eq:alm_h-vector}
\sum_{i=0}^{s}  h_{i}^{\prime} t^{i} = \frac{1}{1-t}\sum_{i=0}^{s}  (h_{s-i}- h_{i}) t^{i}.
\end{equation}
From Theorem~\ref{thm:main}, we have $s= \deg h(\mathbb{K}[\mathfrak{g}_{n}]; t)=\sum_{j=1}^{m} jr_{j} =N$.
For $f(t)=\sum_{i=0}^{s}a_{i}t^{i}$, we have $t^{s}f(t^{-1})=\sum_{i=0}^{s}a_{s-i}t^{i}$.
Applying this to \eqref{eq:alm_h-vector}, we can express 
\[\sum\limits_{i=0}^{s}  h_{i}^{\prime} t^{i} = \frac{t^{N}h(\mathbb{K}[\mathfrak{g}_{n}]; t^{-1})- h(\mathbb{K}[\mathfrak{g}_{n}]; t)}{1-t}.\]
By \eqref{eq:hvector}, we have \[t^{N}h(\mathbb{K}[\mathfrak{g}_{n}]; t^{-1})= t^{N}\prod_{j=1}^m(1+\dots +t^{-j})^{r_j} - t^{N-1}\prod_{j=1}^m(1+\dots + t^{-(j-1)})^{r_j}.\]
For $N=\sum\limits_{j=1}^{m}jr_{j}$, we can express $N-1$ as $\sum\limits_{j=1}^{m}(j-1)r_{j} + \sum\limits_{j=1}^{m}r_{j} -1 $. Therefore,
\begin{align*}
  t^{N}h(\mathbb{K}[\mathfrak{g}_{n}]; t^{-1}) &=  t^{N}\prod_{j=1}^{m} (1+\dots +t^{-j})^{r_j} - t^{n-1}\prod_{j=1}^m t^{(j-1)r_{j}}\big(1+\dots + t^{-(j-1)}\big)^{r_j}\\
  &= \prod_{j=1}^m(1+\dots +t^{j})^{r_j} - t^{n-1}\prod_{j=1}^m(1+\dots + t^{j-1})^{r_j}.
\end{align*}
Hence, we have
\begin{align*}
t^{N}h(\mathbb{K}[\mathfrak{g}_{n}]; t^{-1})- h(\mathbb{K}[\mathfrak{g}_{n}]; t) & = \prod_{j=1}^m(1+\dots +t^{j})^{r_j} - t^{n-1}\prod_{j=1}^m(1+\dots + t^{j-1})^{r_j}\\
& \hspace{0.4cm}- \prod_{j=1}^m(1+\dots +t^{j})^{r_j} + t\prod_{j=1}^m(1+\dots + t^{j-1})^{r_j} \\
& = t(1-t^{n-2}) \prod_{j=1}^m(1+\dots + t^{j-1})^{r_j} . 
\end{align*}

Consequently, \eqref{eq:alm_h-vector} becomes
\begin{align*}
  \sum\limits_{i=0}^{s}  h_{i}^{\prime} t^{i}  & =   \frac{t(1-t^{n-2}) \prod\limits_{j=1}^m(1+\dots + t^{j-1})^{r_j} }{1-t}\\
  &= t (1+t+\cdots +t^{n-3})\prod\limits_{j=1}^m(1+\dots + t^{j-1})^{r_j}. 
\end{align*}
We see that \eqref{eq:alm_h-vector} evaluated at $t=1$ gives $\tilde{e}(\mathbb{K}[\mathfrak{g}_{n}])$.
Thus, $\tilde{e}(\mathbb{K}[\mathfrak{g}_{n}]) = (n-2) \prod\limits_{j=1}^{m}j^{r_j} $.
Recall that for $n\geq 3$, the edge ring $\mathbb{K}[\mathfrak{g}_{n}]$ is almost Gorenstein if and only if $\tilde{e}(\mathbb{K}[\mathfrak{g}_{n}])= n-2$.
That is when $\prod\limits_{j=1}^{m}j^{r_j}= 1$. This occurs exclusively when $r_{j}=0$ for all $j>1$. 
In other words, every cycle in $\mathfrak{g}_{n}$ is a $3$-cycle ($N=n$). 
\end{proof}

From Theorem~\ref{thm:almostGoren} in conjunction with \cite{HN}*{Theorem 1.1}, the following result on $\mathbb{K}[\mathfrak{g}_{r_1, \dots, r_m}]$ is derived.

\begin{corollary}\label{coro:Gorenstein}
For $\mathfrak{g}_{r_1,\dots ,r_m}$ with $n=\sum\limits_{j=1}^{m}r_{j}$ and $N=\sum\limits_{j=1}^{m} jr_{j}$, the edge ring $\mathbb{K}[\mathfrak{g}_{r_1, \dots, r_m}]$ is almost Gorenstein but not Gorenstein if and only if $n\geq3$ and $N=n$. 
\end{corollary}

\section*{Acknowledgements}
Work on this project began while the first named author was a JSPS International Research Fellow at Osaka University, with support provided by JSPS and Mitacs. The second named author is partially supported by KAKENHI 20K03513 and 21KK0043.

\bibliography{main}

\end{document}